\documentclass[11pt]{article}
\usepackage{epsfig,enumerate,amsmath,amsfonts,amsbsy,amssymb,amsthm,mathrsfs,lineno,ifpdf}
\usepackage[mathscr]{euscript}
\usepackage[numbers]{natbib}
\usepackage{hyperref, url}
\usepackage{setspace,graphicx}
\usepackage[usenames,dvipsnames]{pstricks}
\usepackage[margin=1.2in]{geometry}
\newtheorem{theorem}{Theorem}
\newtheorem{lemma}{Lemma}

\newtheorem{corollary}{Corollary}

\newtheorem{conjecture}{Conjecture}

\newtheorem{problem}{Problem}
\newtheorem{claim}{Claim}

\usepackage[normalem]{ulem}
\usepackage{enumitem}
\usepackage{etoolbox}

\let\oldenumerate\enumerate
\renewcommand{\enumerate}{
	\oldenumerate
	\setlength{\itemsep}{1.5pt}
	\setlength{\parskip}{0pt}
	\setlength{\parsep}{0pt}
}

\numberwithin{equation}{section}

\hypersetup{
	colorlinks=true,
	linkcolor=blue,
	filecolor=magenta,
	pdfpagemode=FullScreen,
}
\begin{document}
	
	\title{Perfect divisibility and perfect-Pollyanna in bull-free graphs}
	\author{Ran Chen$^{1}$\footnote{Email: nnuchen@foxmail.com}, \;Paras Vinubhai Maniya$^{2}$\footnote{Corresponding author; Email: maniyaparas9999@gmail.com}, \; Di Wu$^{3}$\footnote{Email: diwu@njit.edu.cn}, \; Junran Yu$^{1}$\footnote{Email: nnuyu@foxmail.com}\\
		\small $^1$Institute of Mathematics, School of Mathematical Sciences\\
		\small Nanjing Normal University, 1 Wenyuan Road,  Nanjing, 210023,  China\\
		\small $^2$Department of Mathematics and Computing\\
		\small Indian Institute of Technology (ISM), Dhanbad, India\\
		\small $^3$Department of Mathematics and Physics\\
		\small Nanjing Institute of Technology, Nanjing 211167, Jiangsu, China\\
	}
	\date{}
	\maketitle

	\begin{abstract}
	A graph $G$ is {\em perfectly divisible} if, for each induced subgraph $H$ of $G$, $V(H)$ can be partitioned into $A$ and $B$ such that $H[A]$ is perfect and $\omega(H[B])<\omega(H)$. A {\em bull} is a graph consisting of a triangle with two disjoint pendant edges. Ho\`ang [Discrete Math. 349 (2026) 114809] proposed four conjectures: 1. $P_5$-free graphs are perfectly divisible; 2. Odd hole-free graphs are perfectly divisible; 3. Even hole-free graphs are perfectly divisible; and 4. $4K_1$-free graphs are perfectly divisible. Karthick {\em et al.} [Electron. J. Combin. 29 (2022) P3.19] proposed a conjecture: Fork-free graphs are perfectly divisible. In this paper, we prove that all of five conjectures above hold for bull-free graphs. Our results also generalize some results of Chudnovsky and Sivaraman [J. Graph Theory 90 (2019) 54--60] and Karthick {\em et al.} [Electron. J. Combin. 29 (2022) P3.19].
	
	 We say that a class ${\cal C}$ is {\em perfect-Pollyanna} if ${\cal C}\cap {\cal G}$ is perfectly divisible for any hereditary class ${\cal G}$ in which each triangle-free graph  is 3-colorable. Let $H\in\{\text{house, hammer, diamond}\}$. In this paper, we prove that the class of $(\text{bull}, H)$-free graphs is perfect-Pollyanna. Let ${\cal C}$ be the class of $(\text{bull}, H)$-free graphs. This implies that ${\cal C}\cap {\cal G}$ is perfectly divisible if and only if all of triangle-free graphs in ${\cal G}$ are perfectly divisible.  As corollaries, we show that $(\text{bull},{\cal H})$-free graphs are perfectly divisible, where ${\cal H}$ is one of $\{P_{11},C_4\},\{P_{14},C_5,C_4\}$, and $\{P_{17},C_6,C_5,C_4\}$.
	\end{abstract}

	\begin{flushleft}
	\end{flushleft}
	
	\newpage
	
	\section{Introduction}\label{introdction}

In this paper, we consider only finite and simple graphs. We follow \cite{BM08} for basic graph-theoretic notation and terminology. Let $P_k$ and $C_k$ be a {\em path} and a {\em cycle} on $k$ vertices respectively. We say that a graph $G$ {\em contains} a graph $H$ if $H$ is isomorphic to an induced subgraph of $G$. A graph $G$ is $H$-{\em free} if it does not contain $H$. Analogously, for a family $\mathcal{H}$ of graphs, we say that $G$ is $\mathcal{H}$-free if no member of $\mathcal{H}$ is isomorphic to an induced subgraph of $G$ .

For a given positive integer $k$, we use $[k]$ to denote the set $\{1,\ldots, k\}$. A {\em k-coloring} of a graph $G=(V(G),E(G))$ is a mapping $f$: $V(G)\rightarrow [k]$ such that $f(u)\neq f(v)$ whenever $uv\in E(G)$. We say that $G$ is {\em k-colorable} if $G$ admits a $k$-coloring. The {\em chromatic number} of $G$, denoted by $\chi(G)$, is the smallest positive integer $k$ such that $G$ is $k$-colorable. A {\em clique} (resp. {\em stable set}) of $G$ is a set of pairwise adjacent (resp. non-adjacent) vertices in $G$. The {\em clique number} of $G$, denoted by $\omega(G)$, is the maximum cardinality among all cliques of  $G$. 
	
The concept of binding functions was introduced by Gy\'{a}rf\'{a}s \cite{G75} in 1975. A class of graphs is said to be {\em hereditary} if it is closed under isomorphism and induced  subgraphs. For a hereditary class $\mathcal{F}$, if there exists a function $f$ such that $\chi(H)\le f(\omega(H))$ for every $H\in\mathcal{F}$, then $\mathcal{F}$ is called \emph{$\chi$-bounded} and $f$ a \emph{binding function}.
	
An induced cycle of length $k\ge 4$ is called a {\em hole}, and $k$ is the {\em length} of the hole. A {\em k-hole} is a hole of length $k$. A hole is {\em odd} if $k$ is odd, and {\em even} otherwise. An {\em antihole} is the complement graph of a hole. A graph $G$ is said to be {\em perfect} if $\chi(H)=\omega(H)$ for every induced subgraph $H$ of $G$. The Strong Perfect Graph Theorem \cite{CRST06} was  established by Chudnovsky {\em et al.} in 2006.
	
\begin{theorem}[\cite{CRST06}]\label{perfect}
A graph $G$ is perfect if and only if $G$ is $($odd hole, odd antihole$)$-free.
\end{theorem}

	\begin{figure}[htbp]
		\begin{center}
			\includegraphics[width=17cm]{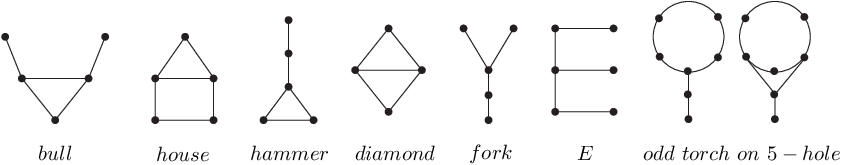}
		\end{center}
		\vskip -25pt
		\caption{Illustration of some small graphs.}
		\label{fig-1}
	\end{figure}
	
A {\em bull} is a graph consisting of a triangle with two disjoint pendant edges.  A {\em house} is the complement of $P_5$, a {\em hammer} is a graph obtained by identifying an end vertex of a $P_3$ with a vertex of a triangle. A {\em diamond} is a graph consisting of two triangles sharing exactly one common edge. A {\em fork} is a graph obtained from $K_{1,3}$ by subdividing an edge once. An {\em E} is a graph obtained from an induced $P_5=v_1v_2v_3v_4v_5$ by adding an new vertex adjacent to only $v_3$ in the $P_5$. An {\em odd torch} is a graph obtained from an odd hole by adding an edge such that one endpoint is not adjacent to any vertex of the odd hole, and the other endpoint is adjacent to the vertices of a stable set of the odd hole, for example, odd torch on a 5-hole is shown in Figure~\ref{fig-1}.

 	Let the vertices of $V(G)$ be partitioned into two sets $A$ and $B$. We say that $(A,B)$ is a {\em good} partition of $G$ if $G[A]$ is perfect and $\omega(G[B])<\omega(G)$. A graph $G$ is {\em perfectly divisible} if for every induced subgraph $H$ contains a good partiton. This concept was proposed by Ho\`{a}ng~\cite{H22}. By a simple induction on $\omega(G)$, every perfectly divisible graph $G$ satisfies $\chi(G)\le\binom{\omega(G)+1}{2}$ \cite{CS19}. If $G$ is non-perfectly divisible but every proper induced subgraph of $G$ is perfectly divisible, then $G$ is called {\em minimally non-perfectly  divisible}, abbreviated as MNPD. 
 	
	Let $G$ be a graph and let $w$ be a weight function on $V(G)$. We use {\em $\omega_w(G)$} to denote the maximum weight of a clique in $G$. A graph $G$ is {\em perfectly weight divisible} \cite{CS19,H22} if for every positive integer weight function $w$ on $V(G)$, and for every induced subgraph $H$ of $G$, there is a partition of $V(H)$ into $P$ and $W$ such that $H[P]$ is perfect and $\omega_w(H[W])<\omega_w(H)$.  If $G$ is non-perfectly weight divisible but every proper induced subgraph of $G$ is perfectly weight divisible, then $G$ is called {\em minimally non-perfectly weight divisible}, abbreviated as MNWD. 
		
	It is certain that if a graph $G$ is perfectly weight divisible, then $G$ is perfectly divisible. We say that a class of graphs ${\cal G}$ is perfectly weight divisible (resp. perfectly  divisible) if every graph in ${\cal G}$ is perfectly weight divisible (resp. perfectly divisible).
		
	Perfectly divisibility has been widely studied and discussed for a long time.  Ho\`{a}ng~\cite{H22} proved that each (banner, odd hole)-free graph is perfectly divisible, where banner is a graph obtained from $C_4$ by attaching a pendant edge to a vertex of $C_4$. Recently, Ho\`{a}ng~\cite{H26} studied the structure underlying perfect divisibility and proposed four important and interesting conjectures.
	\begin{conjecture}[\cite{H26}]\label{con-1}
		Every $P_5$-free graph is perfectly divisible.
	\end{conjecture}
	\begin{conjecture}[\cite{H26}]\label{con-2}
		Every odd hole-free graph is perfectly divisible.
	\end{conjecture}
	\begin{conjecture}[\cite{H26}]\label{con-3}
		Every even hole-free graph is perfectly divisible.
	\end{conjecture}
	\begin{conjecture}[\cite{H26}]\label{con-4}
		Every $4K_1$-free graph is perfectly divisible.
	\end{conjecture}
	
		Karthick {\em et al.}~\cite{KKS22} studied the perfect divisibility of fork-free graphs and proved that some subclasses of fork-free graphs are perfect divisible. Moreover, they proposed the following conjecture, which remains open.
		
	\begin{conjecture}[\cite{KKS22}]\label{con-5}
		Every fork-free graph is perfectly divisible.
	\end{conjecture}

		Chudnovsky and Sivaraman~\cite{CS19} proved that every (bull, odd hole)-free graph and every (bull, $P_5$)-free graph is perfectly divisible. Karthick {\em et al.}~\cite{KKS22} proved that every (bull, fork)-free graph is perfectly divisible. In this paper, we show that every (bull, odd torch)-free graph is perfectly divisible. Notice that every odd torch contains a $P_5$, odd hole, fork and the graph E, and so the result generalizes their results. This implies that Conjectures~\ref{con-1}, \ref{con-2}, and \ref{con-5} hold for bull-free graphs. 
		
			\begin{theorem}\label{oddballoon}
			Let $G$ be a $(bull, odd~torch)$-free graph. Then $G$ is perfectly weight divisible, and hence $G$ is perfectly divisible.
		\end{theorem}
		
		We also show that Conjectures~\ref{con-3} and \ref{con-4} holds for bull-free graphs as following. 
		
		\begin{theorem}\label{evenhole}
			Let $G$ be a  $($bull, even hole$)$-free graph. Then $G$ is perfectly weight divisible, and hence $G$ is perfectly divisible.
		\end{theorem}
		
		\begin{theorem}\label{4K1}
			Let $G$ be a $(bull, 4K_1)$-free graph. Then $G$ is perfectly weight divisible, and hence $G$ is perfectly divisible.
		\end{theorem}

		 A class of graphs is {\em polynomially} (resp. {\em linearly}) $\chi$-{\em bounded} if it has a polynomial (resp. linear) binding function. Esperet \cite{E2017} conjectured that every $\chi$-bounded class of graphs is polynomially $\chi$-bounded. However, this conjecture was recently disproved by
		 Bria\'{n}ski {\em et al.} \cite{BDW2023}. Inspired by Esperet's conjecture,
		 Chudnovsky {\em et al.} \cite{CCDO2023} considered its analog for proper classes of graphs, and introduced the concepts of {\em Pollyanna} as follows: A class ${\cal C}$ of graphs is {\em Pollyanna} if ${\cal C}\cap {\cal F}$ is polynomially $\chi$-bounded for every $\chi$-bounded class ${\cal F}$ of graphs; moreover, they showed that the class of bull-free graphs is Pollyanna. Recently, Chen and Xu \cite{CX2024} proved that the class of (bull, $C_4$)-free graphs is linear-Pollyanna (A class ${\cal C}$ of graphs is {\em linear-Pollyanna} if
		 ${\cal C}\cap {\cal F}$ is linearly $\chi$-bounded for every $\chi$-bounded class ${\cal F}$ of graphs.).

		From all of Theorems above, it is interesting to study the perfect divisibility of bull-free graphs.
		 Note that every bull-free graph is not perfectly divisible as the class of bull-free graphs is not $\chi$-bounded by Erd\H{o}s \cite{E1959} who proved that the class of triangle-free graphs is not $\chi$-bounded. A natural question is how close the class of bull-free graphs is to being perfectly divisible. In particular, for which graph classes does the intersection with the class of bull-free graphs become perfectly divisible?
		 
		 Let ${\cal C}$ be a class of graphs. We say that ${\cal C}$ is {\em perfect-Pollyanna} if ${\cal C}\cap {\cal G}$ is perfectly divisible for any hereditary graphclass ${\cal G}$ in which each triangle-free graph  is $3$-colorable. It is certain that if ${\cal C}$ is perfectly divisible, then ${\cal C}$ is perfect-Pollyanna. All of Conjectures~\ref{con-1}-\ref{con-5} can be restated that these classes of graphs are perfect-Pollyanna. In this paper, we study the perfect divisibility of bull-free graphs and prove the following theorems.
		 
	\begin{theorem}\label{pollyanna-1}
		Let ${\cal C}$ be the class of $(\text{bull}, H)$-free graphs, where $H\in \{\text{house, hammer, diamond}\}$. Then ${\cal C}$ is perfect-Pollyanna.
	\end{theorem}
	
	Ho\`{a}ng (Theorem 3.2 in \cite{H26}) proved that a triangle-free graph $G$ is perfectly divisible if and only if $\chi(G)\leq3$. So Theorem~\ref{pollyanna-1} can be immediately obtained from the following theorem. 
		
	\begin{theorem}\label{pollyanna-2}
		Let ${\cal C}$ be the class of $(\text{bull}, H)$-free graphs, where $H\in \{\text{house, hammer, diamond}\}$, ${\cal G}$ be a hereditary class of graphs, and ${\cal H}$ be the class of triangle-free graphs. Then ${\cal G}\cap {\cal C}$ is perfectly divisible if and only if  ${\cal G}\cap {\cal H}$ is perfectly divisible.
	\end{theorem}
	
	Let $F$ be a graph. By applying Theorem~\ref{pollyanna-2} with ${\cal G}$ to be the class of $F$-free graphs or $(F,C_4)$-free graphs, we can immediately obtain the following corollaries.
	
	\begin{corollary}\label{co1}
		Let $F$ be a graph. If $(F,C_3)$-free graph is $3$-colorable, then $(F,bull,H)$-free graph is perfectly divisible, where $H\in \{\text{house, hammer, diamond}\}$.
	\end{corollary}
	
	\begin{corollary}\label{co3}
		Let $F$ be a graph. If $(F,C_4,C_3)$-free graph is $3$-colorable, then $(F,C_4,bull)$-free graph is perfectly divisible.
	\end{corollary}
	
	  Goedgebeur and Schaudt \cite{G} proved that all of $(P_{11}, C_4, C_3)$-free graphs, $(P_{14},C_5,C_4,C_3)$-free graphs, and $(P_{17},C_6,C_5,C_4,C_3)$-free graphs are 3-colorable. By Theorem~\ref{pollyanna-2}, we can obtain the following corollary.
	 
	 \begin{corollary}\label{co2}
	 	The three classes of $(bull,P_{11},C_4)$-free graphs, $(bull,P_{14},C_5,C_4)$-free graphs and\\ $(bull,P_{17},C_6,C_5,C_4)$-free graphs are all perfectly divisible.
	 \end{corollary}

	Notice that the class of $3K_1$-free graphs has no linear binding function \cite{BGS2002,SR2019}.  This implies that the class of 
	$(\text{bull}, H)$-free graphs, where $H\in\{\text{odd~torch}, 4K_1\}$, do not admit a linear binding function. According to Theorem~\ref{oddballoon}, we can directly derive the following corollaries. 
	
	\begin{corollary}\label{chromatic}
		Let $G$ be a $(\text{bull}, H)$-free graph, where $H\in\{\text{odd~torch},4K_1\}$. Then $\chi(G)\le \binom{\omega(G)+1}{2}$.
	\end{corollary}
	
	\begin{corollary}\label{co-2}
		Let $G$ be an $(\text{odd torch}, C_3)$-free graph. Then $\chi(G)\leq3$. 
	\end{corollary}

 As mentioned above, the structure and coloring problems of bull-free graphs have garnered significant attention, with new methods and results \cite{CW2025,CCDO2023,DC2025,HPPS,SR2019} being continuously proposed. This paper utilizes the concept of perfectly weight divisibility to further analyze the structure of bull-free graphs and bull-free MNWD graphs, thereby deriving our main results. In Section~\ref{20}, we will give some structural properties of bull-free graphs and prove Theorem~\ref{oddballoon}.  We will prove Theorems~\ref{evenhole} and \ref{pollyanna-2} in Section~\ref{3}, and prove Theorem~\ref{4K1} in Section~\ref{4}.

	\section{Notations and preliminary results}\label{notations}

For $X\subseteq V(G)$, we use $G[X]$ to denote the subgraph of $G$ induced by $X$. Let $v\in V(G)$, $X\subseteq V(G)$. We use $N_G(v)$ to denote the set of vertices adjacent to $v$. Let $d_G(v)=|N_G(v)|$, $M_G(v)=V(G)\setminus (N_G(v)\cup\{v\})$, $N_G(X)=\{u\in V(G)\setminus X\;|\; u$ has a neighbor in $X\}$, and $M_G(X)=V(G)\setminus (X\cup N_G(X))$. We simply write $N(v)$, $d(v)$, $M(v)$, $N(X)$ and $M(X)$ to denote $N_G(v)$, $d_G(v)$, $M_G(v)$, $N_G(X)$ and $M_G(X)$, respectively, whenever no confusion arises.

For two vertices $u$ and $v$, let $d(u, v)$ be the distance of $u$ and $v$, which is the length of a shortest path between $u$ and $v$. Let $S_1,S_2\subseteq V(G)$. We define $d(S_1,S_2)=\min\{d(s_1,s_2)~|~s_1\in S_1,s_2\in S_2\}$. For $A, B\subseteq V(G)$, let $N_A(B)=N(B)\cap A$ and $M_A(B)=A\setminus (N_A(B)\cup B)$.  For $u, v\in V(G)$, we simply write $u\sim v$ if $uv\in E(G)$, and write $u\not\sim v$ if $uv\not\in E(G)$. For a set $A\subset V(G)$ and a vertex $u\in V(G)\setminus A$, we say that $u$ is \emph{complete} to $A$ if $u$ is adjacent to every vertex of $A$, and that $u$ is \emph{anticomplete} to $A$ if $u$ is not adjacent to any vertex of $A$. For two disjoint subsets $A$ and $B$ of $V(G)$, $A$ is complete to $B$ if every vertex of $A$ is complete to $B$, and $A$ is anticomplete to $B$ if every vertex of $A$ is anticomplete to $B$.

Let $S\subseteq V(G)$ with $1<|S|<|V(G)|$. We say that $S$ is {\em a homogeneous set } of $G$ if every vertex in $V(G)\setminus S$ is either complete to $S$ or anticomplete to $S$. We need the following lemmas.

\begin{lemma}{ \em(Theorem 3.6 in \cite{CS19})}\label{homogeneous}
	Every MNWD graph has no homogeneous set. 
\end{lemma}

\begin{lemma}{\em (4.3 in \cite{CS2008})}\label{bullfree}
	If $G$ is a bull-free graph, then either $G$ has a homogeneous set or for every $v\in V(G)$, either $G[N(v)]$ is perfect or $G[M(v)]$ is perfect.
\end{lemma}

Recall that Ho\`ang (Theorem 3.2 in  \cite{H26}) proved that a triangle-free graph $G$ is perfectly divisible if and only if $\chi(G)\leq3$. We now prove a stronger statement.	
	
	\begin{lemma}\label{3-coloring}
	Let $G$ be a triangle-free graph. Then the following results are equivalent to each other.\\[-24pt]
	\begin{enumerate}
		\item  $\chi(G)\leq3$;
		\item $G$ is perfectly weight divisible;
		\item $G$ is perfectly divisible.
	\end{enumerate}
\end{lemma}
\begin{proof}
It is certain that (i) is equivalent to (iii), and (iii) can be obtained from (ii). It suffices to prove that (ii) can be obtained from (i). Let $w$ be a positive integer weight function on $V(G)$, and let $G'$ be an induced subgraph of $G$. Since $\chi(G')\leq3$, there is a partition $(A,B)$ of $V(G')$ such that $G'[A]$ is bipartite and $B$ is a stable set of $G'$. Let $X\subseteq B$ such that each vertex in $X$ is an isolated vertex of $G'$. Then $(A\cup X, B\setminus X)$ is a partition of $G'$ such that $G[A\cup X]$ is perfect, and $\omega_w(G[B\setminus X])<\omega_w(G')$. This proves the lemma~\ref{3-coloring}.
\end{proof}

 As standard notation, we use $\delta(G)$ to denote the minimum degree of $G$. The {\em Cartesian product} of any two graphs $G$ and $H$, denoted by $G\square H$, is the graph with vertex set $\{(a,u)~|~a\in V(G)~\mbox{and}~u\in V(H)\}$, where two vertices $(a,u)$ and $(b,v)$ are adjacent for $a,b\in V(G)$ and $u,v\in V(H)$ if either $a=b$ and $u\sim v$ in $H$, or $u=v$ and $a\sim b$ in $G$. 

Chen and Xu  \cite{CX2024} proved the following lemma.

\begin{lemma}{\em (Theorem 3.2 in \cite{CX2024})}\label{diamond-1}
	Let $G$ be a connected $(\text{bull, diamond})$-free graph. Then either $G$ is triangle-free or $\delta(G)\leq\omega(G)-1$ or $G$ is isomorphic to $K_2\square K_{\omega(G)}$. 
\end{lemma}

\begin{theorem}\label{diamond}
	Every $($bull, diamond$)$-free  MNPD  graph is triangle-free.
\end{theorem}
\begin{proof}
Suppose to its contrary. Then we have that either $\delta(G)\leq\omega(G)-1$ or $G$ is isomorphic to $K_2\square K_{\omega(G)}$. By Theorem \ref{perfect}, $K_2\square K_{\omega(G)}$ is perfect, which implies that $\delta(G)\leq\omega(G)-1$. Let $v\in V(G)$ such that $d(v)\leq\omega(G)-1$. 

Since $G$ is MNPD, it follows that there exists a partition $(X,Y)$ of $G[V(G)\setminus\{v\}]$ such that $G[X]$ is perfect and $\omega(G[Y])<\omega(G[V(G)\setminus\{v\}])$. Then $v$ belongs to a maximum clique of $G$ as otherwise $(X, Y\cup \{v\})$ is a good partition of $G$ as $\omega(G[Y\cup \{v\}])<\omega(G)$, a contradiction. So $N(v)$ is a clique as $d(v)\leq\omega(G)-1$. However, $(X\cup \{v\},Y)$ is a good partition of $G$ as $G[X\cup \{v\}]$ is perfect, a contradiction. This proves Theorem~\ref{diamond}.
\end{proof}

	\section{Structural properties of bull-free graphs}\label{20}	
	
	 We say that a graph $G$ is {\em locally perfect} if for every vertex $v\in V(G)$, $G[N(v)]$ is perfect. In this section, we first give some structural properties of bull-free graphs. Finally, we will prove Theorem~\ref{oddballoon}.
	
		\begin{lemma}\label{oddantihole}
		Let $G$ be a connected locally perfect bull-free graph, and let $v\in V(G)$. Then $G[M(v)]$ does not contain an odd antihole with at least $7$ vertices.
	\end{lemma}
	\begin{proof}
	Suppose to the contrary that $H$ is an odd antihole of $G[M(v)]$ with $V(H)=\{v_1,v_2,...,v_k\}$, where $k\ge 7$ and $\overline{H}=v_1v_2\cdots v_kv_1$ is an odd hole. We first prove that
	\begin{equation}\label{c-1}
		\mbox{if $y\in N(V(H))$, then $|N(y)\cap V(H)|\geq2$.}
	\end{equation}
	
	On the contrary, without loss of generality, let $N(y)\cap V(H)=\{v_1\}$. However, $\{v_1,v_3,v_5,y,v_2\}$ induces a bull, a contradiction. This proves (\ref{c-1}). 
	
	We next prove that 
	\begin{equation}\label{c-2}
		\mbox{if $x\in M(V(H))$ and $y\in N(V(H))$ such that $x\sim y$, then $N(y)\cap V(H)$ is a stable set.}
	\end{equation}
	
	On the contrary, without loss of generality, assume that $v_1,v_n\in N(y)\cap V(H)$ with $v_1v_n\in E(G)$, where $3\leq n\leq k-1$. We next show that $y$ is complete to $\{v_1,v_2,\cdots,v_n\}$. For the sake of contradiction, assume that $y\not\sim v_{n'}$ such that $n'$ is the minimum integer, where $2\leq n'\leq n-1$. Suppose that $n=3$. Then $n'=2$. To avoid an induced bull on $\{v_1,v_3,y,x,v_4\}$, we must have $y\sim v_{4}$; however, $\{v_1,v_4,y,x,v_2\}$ induces a bull, a contradiction. Hence, $n\geq4$, and so $v_n\sim v_2$ and $v_{n-1}\sim v_{1}$. 
	
	We can deduce that $n'\ne 2$ to avoid an induced bull on $\{v_1,v_n,y,x,v_2\}$; and $n'\ne n-1$ to avoid an induced bull on $\{v_1,v_n,y,x,v_{n-1}\}$. Therefore, $3\leq n'\leq n-2$, and thus $v_{n'}\sim v_n$; moreover $y\sim v_{n'-1}$ by the minimality of $n'$. However, $\{v_{n'-1},v_n,y,x,v_{n'}\}$ induces a bull, a contradiction. So, $y$ is complete to $\{v_1,v_2,\cdots,v_n\}$. Using similar arguments, we can deduce that $y$ is complete to $\{v_1,v_k,v_{k-1},\cdots,v_n\}$, and this implies that $y$ is complete to $V(H)$, which contradicts that $G$ is locally perfect. This proves (\ref{c-2}).

	Since $G$ is connected and $V(H)\subseteq M(v)$, there exist two adjacent vertices $x\in M(V(H))$ and $y\in N(V(H))$. Using (\ref{c-1}) and (\ref{c-2}), we may assume by symmetry that $N(y)\cap V(H)=\{v_1,v_2\}$. However, $\{v_{k-3},v_{k-1},v_k,v_1,y\}$ induces a bull, a contradiction. This completes the proof of Lemma~\ref{oddantihole}.
	\end{proof}
	\begin{lemma}\label{MNWD}
		Let $G$ be a bull-free MNWD graph. Then $G$ is a connected locally perfect graph, and for every vertex $v\in V(G)$, $G[M(v)]$ is imperfect and contains no odd antihole with at least $7$ vertices.
	\end{lemma}
	\begin{proof}
	It is certain that $G$ is connected. Since $G$ is MNWD, we have that for every vertex $v\in V(G)$, $G[M(v)]$ cannot be perfect as otherwise, $(\{v\}\cup M(v), N(v))$ is a partition of $V(G)$ such that for all positive integer weight function on $V(G)$, the $G[\{v\}\cup M(v)]$ is a perfect graph and $G[N(v)]$ has a smaller maximum weight of a clique than $G$, a contradiction to the minimality of $G$. Therefore, for every $v\in V(G)$, $G[M(v)]$ is imperfect. Moreover, since $G$ is MNWD, $G$ has no homogeneous set by Lemma~\ref{homogeneous}. Hence, for every $v\in V(G)$, $G[N(v)]$ is perfect by Lemma~\ref{bullfree}. Now, Lemma~\ref{MNWD} holds by Lemma~\ref{oddantihole}.
	\end{proof}

	\begin{lemma}\label{L-11}
		Let $G$ be a locally perfect bull-free graph, $C=v_1v_2\cdots v_kv_1$ be an odd hole in $G$, where $k\ge5$ is an odd integer. Let  $uv$ be an edge of $G$. If $u\in N(V(C))$ and $v\in M(V(C))$, then $N_C(u)$ is stable.
	\end{lemma}
	\begin{proof}
	On the contrary, without loss of generality, we may assume that $u$ is complete to $\{v_1,v_2\}$. Since $G$ is locally perfect, $u$ is not complete to $V(C)$. Let $n\in \{3,4,\cdots,k\}$ be the minimum integer such that $u\not\sim v_n$. Then $\{u,v_{n-1},v_{n-2},v_n,v\}$ induces a bull, a contradiction. 	\end{proof}

\begin{lemma}\label{P-1}
	Let $G$ be a locally perfect bull-free graph, $C=v_1v_2\cdots v_kv_1$ be an odd hole in $G$, where $k\ge5$ is an odd integer. Let $x,y,z\in V(G)$ such that $x\sim y$, $\{x,y\}\subseteq N(z)\cap N(V(C))$ and $z\in M(V(C))$. If both $x$ and $y$ have at least two neighbors in $V(C)$, then $N_C(x)=N_C(y)$.
\end{lemma}
\begin{proof}
Suppose to the contrary that $N_C(x)\ne N_C(y)$. Without loss of generality, assume that $x\sim v_1$ and $y\not\sim v_1$. We next show that 
\begin{equation}\label{eq-1}
	\mbox{$y$ has a neighbor in $\{v_2,v_n\}$.}
\end{equation} 

Suppose to the contrary that $y$ is anticomplete to $\{v_2,v_n\}$. Since $y$ has a neighbor in $V(C)$, let $n\in \{3,4,\cdots,k-1\}$ be the minimum integer such that $y\sim v_n$. By Lemma~\ref{L-11}, $y$ is anticomplete to $\{v_{n-1},v_{n+1}\}$. To avoid an induced bull on $\{x,y,z,v_1,v_n\}$, we have that $x\sim v_n$. Therefore, $x$ is anticomplete to $\{v_{n-1},v_{n+1}\}$ by Lemma~\ref{L-11} again. If $n=3$, then $\{x,y,v_3,v_1,v_4\}$  induces a bull, a contradiction. If $n\geq4$, then $\{x,y,v_1,v_{n-1},v_n\}$  induces a bull, a contradiction. This proves (\ref{eq-1}). By symmetry, we may assume that $y\sim v_2$. By Lemma~\ref{L-11}, $x$ is anticomplete to $\{v_2,v_k\}$ and $y\not\sim v_3$.

Suppose $y$ is not anticomplete to $\{v_4,v_5,\cdots,v_{k-1}\}$. Let $m\in \{4,\cdots,k-1\}$ be the minimum integer such that $y\sim v_m$. To avoid an induced bull on $\{x,y,z,v_1,v_m\}$, we have that $x\sim v_m$. By Lemma~\ref{L-11}, $x\not\sim v_{m-1}$. However, $\{x,y,v_m,v_1,v_{m-1}\}$ induces a bull, a contradiction. Hence, $y$ is anticomplete to $\{v_4,v_5,\cdots,v_{k-1}\}$. Similarly, $x$ is anticomplete to $\{v_4,v_5,\cdots,v_{k-1}\}$. Since both $x$ and $y$ have at least two neighbors in $V(C)$, $N_C(x)=\{v_1,v_3\}$ and $N_C(y)=\{v_2,v_k\}$. However, $\{x,y,z,v_3,v_k\}$ induces a bull. This completes the proof of Lemma~\ref{P-1}.
\end{proof}

We are now in a position to present the proof of Theorem~\ref{oddballoon}.

\noindent\textbf{{\em Proof of Theorem~\ref{oddballoon}:}} For the sake of contradiction, let $G$ be an (bull, odd torch)-free MNWD graph. By Lemma~\ref{MNWD}, we have that $G$ is connected locally perfect graph, and for every vertex $v\in V(G)$, $G[M(v)]$ is imperfect and contains no odd antihole with at least $7$ vertices. Therefore, by Theorem~\ref{perfect}, $G[M(v)]$ contains an odd hole, say $C=v_1v_2\cdots v_kv_1$, where $k\geq5$ is an odd integer. Since $G$ is connected and $V(C)\subseteq M(v)$, there exist two adjacent vertices $x\in M(V(C))$ and $y\in N(V(C))$. By Lemma~\ref{L-11}, we have that $N(y)\cap V(C)$ is stable, which contradicts that $G$ is odd torch-free. This completes the proof of Theorem~\ref{oddballoon}. \qed

\section{Proof of Theorems~\ref{evenhole} and~\ref{pollyanna-2}}\label{3}

In this section, we present the proofs of Theorems~\ref{evenhole} and \ref{pollyanna-2} together, since both can be derived by combining the results from the preceding text with the three lemmas below. 

\begin{lemma}{\em (1.1 in \cite{CS2023})}\label{P-000}
Each even hole-free graph $G$ satisfies $\chi(G)\le 2\omega(G)-1$, and so $G$ is $3$-colorable if $\omega(G)\le2$. 
\end{lemma}

\begin{lemma}\label{P-2}
	Every $(bull, H)$-free MNWD graph is triangle-free, where $H\in \{\text{house, hammer}\}$.
\end{lemma}
	\begin{proof}
	Let $G$ be a (bull, $H$)-free MNWD graph, where $H\in \{\text{house, hammer}\}$. Suppose to the contrary that $G$ contains a triangle $C_1=xyzx$. Let $t\in \{x,y,z\}$. By Lemma  \ref{MNWD}, we have that $G$ is a connected locally perfect graph, and $G[M(t)]$ is imperfect and contains no odd antihole with at least 7 vertices. By Theorem~\ref{perfect}, $G[M(t)]$ contains an odd hole, say $C_2=v_1v_2\cdots v_kv_1$, where $k\geq5$ is an odd integer. 
	
	It is certain that $V(C_1)\cap V(C_2)=\emptyset$ as $t\in M(V(C_2))$. This implies that there exists a shortest path $P=p_1p_2\cdots p_h$ with $h\geq2$ connecting $C_1$ and $C_2$ such that $p_1\in V(C_1)$ and $p_h\in V(C_2)$. Among all of triangles in $G$ and all of vertices of such triangles, we choose a triangle and a vertex of the triangle such that $|V(P)|$ is minimum, and we still write as $C_1,C_2, P$, and suppose that $t=y$. Notice that $y\in M(V(C_2))$. We next prove that following claim.
	
	\begin{claim}\label{cc-2}
		There does not exist a triangle $abca$ in $G$ such that $a\in M(V(C_2))$ and $\{b,c\}\subseteq N(V(C_2))$. 
	\end{claim}
	\begin{proof}
	On the contrary,
	 let $S$ be the component of $G[N(a)\cap N(V(C_2))]$ which contains the edge $bc$. Clearly $1<|V(S)|<|V(G)|$. We next show that
	\begin{equation}\label{eq-2}
		\mbox{for any two vertices $t_1,t_2\in V(S)$, $N_{C_2}(t_1)=N_{C_2}(t_2)$.}
	\end{equation}
	
	It suffices to prove that for any two adjacent vertices $t_1,t_2\in V(S)$, $N_{C_2}(t_1)=N_{C_2}(t_2)$ as $G[S]$ is connected. Suppose to the contrary that there exists an edge $t_1t_2$ such that $N_{C_2}(t_1)\ne N_{C_2}(t_2)$. By Lemma~\ref{P-1}, we have that at least one of $\{t_1,t_2\}$ has exactly one neighbor in $N(V(C_2))$. Without loss of generality, set
	 $N_{C_2}(t_1)=\{v_1\}$. 
	
 Suppose $H$=hammer. Since $N_{C_2}(t_1)\ne N_{C_2}(t_2)$, let $v_j\in V(C_2)$ such that $t_2\sim v_j$ and $j\ne1$. By Lemma~\ref{L-11}, $t_2$ is anticomplete to $\{v_{j-1},v_{j+1}\}$.  If $j=2$, then $\{t_1,t_2,a,v_2,v_3\}$ induces a hammer; otherwise $\{t_1,t_2,a,v_j,v_{j-1}\}$ induces a hammer, a contradiction.

	 Suppose $H$=house. Let $v_j$ be a neighbor of $t_2$ such that $v_j\in V(C_2)\setminus \{v_1\}$. If $j=2$ or $j=k$, then $\{a,t_1,t_2,v_1,v_j\}$ induces a house; otherwise $\{a,t_1,t_2,v_1,v_j\}$ induces a bull. Hence $t_2\sim v_1$. By Lemma~\ref{L-11}, $\{t_1,t_2\}$ is anticomplete to $\{v_2,v_k\}$. Since  $N_{C_2}(t_1)\ne N_{C_2}(t_2)$, there is a vertex $v_i\in V(C_2)$ such that $t_2\sim v_i$, where $i\notin\{1,2,k\}$. However, $\{t_1,t_2,v_1,v_2,v_i\}$ induces a bull if $v_i\not\sim v_2$ or a house if $v_i\sim v_2$, a contradiction. This proves (\ref{eq-2}). 
	
	By Lemma~\ref{homogeneous}, $V(S)$ is not a homogeneous set of $G$, and so there is a vertex $v\in V(G)\setminus V(S)$ such that $v$ is neither complete to nor anticomplete to $V(S)$. Hence, there are two vertices $t_3,t_4\in V(S)$ such that $vt_3t_4$ is an induced $P_3$. By (\ref{eq-2}), $N_{C_2}(t_3)=N_{C_2}(t_4)$ and so $v\notin V(C_2)$. Without loss of generality, we assume that $v_1$ is complete to $\{t_3,t_4\}$. By Lemma~\ref{L-11}, $\{t_3,t_4\}$ is anticomplete to $\{v_2,v_k\}$. 
	
	Suppose that $H$=hammer. To avoid an induced hammer on $\{t_3,t_4,v_1,v_2,v_3\}$, $v_3$ is complete to $\{t_3,t_4\}$. By Lemma~\ref{L-11}, $v_4$ is anticomplete to $\{t_3,t_4\}$. Similarly, since $G$ is hammer-free, $\{t_3,t_4\}$ is complete to $\{v_5,v_7,\cdots,v_k\}$, which contradicts that $\{t_3,t_4\}$ is anticomplete to $v_k$.  
	
	Suppose that $H$=house.  If $v\not\sim v_1$, then $\{t_3,t_4,v_1,v_2,v\}$ induces a bull whenever $v\not\sim v_2$, and $\{t_3,t_4,v_1,v_2,v\}$ induces a house whenever $v\sim v_2$, both are contradictions. Hence $v\sim v_1$. Recall that $S$ is the component of $G[N(a)\cap N(V(C_2))]$, and so $v\not\sim a$. To avoid an induced bull on $\{v_1,t_3,v,a,v_2\}$, we have that $v\sim v_2$. Similarly, $v\sim v_k$. If $v\not\sim v_3$, then $\{v_1,v_2,v,v_3,t_4\}$ induces a bull whenever $t_4\not\sim v_3$, and $\{v_1,v_2,v,v_3,t_4\}$ induces a house whenever $t_4\sim v_3$, both are contradictions. Hence $v\sim v_3$. Similarly, $v\sim v_{k-1}$. 
	Since $G$ is locally perfect, we have that $k\geq7$ and let $n\in\{4,\cdots,k-2\}$ be the minimum integer such that $v_n\not\sim v$. However, $\{v_{n-2},v_{n-1},v_n,v,v_k\}$ induces a bull. This proves Claim~\ref{cc-2}.
	\end{proof}
	
	By Claim~\ref{cc-2}, $\{x,z\}\cap M(V(C_2))\ne\emptyset$ and so we may by symmetry assume that $z\in M(V(C_2))$.  Without loss of generality, assume that $p_1=x$ and $p_h=v_1$.  Let $S$ be the component of $G[(N(x)\setminus\{p_2\})\cap M(V(C_2))]$ which contains the edge $yz$. Clearly $1<|V(S)|<|V(G)|$. By the choice of $C_1,C_2,P,$ and $y$, we have that 
	\begin{equation}\label{eq-3}
		\mbox{$p_2$ is anticomplete to $V(S)$. }
	\end{equation}

	By Lemma~\ref{homogeneous}, $V(S)$ is not a homogeneous set of $G$. So there exists a vertex $v\in V(G)\setminus V(S)$ satisfying that $v$ has a neighbor and a non-neighbor in $V(S)$. Since $G[S]$ is connected, there exist $t_1,t_2\in V(S)$ such that $vt_1t_2$ is an induced $P_3$. 
	
	Suppose that $H$=hammer. If $h\geq3$, then $\{t_1,t_2,x,p_2,p_3\}$ induces a hammer by the choice of $P$. If $h=2$, then $x\sim v_1$ and so $x\not\sim v_2$ by Lemma~\ref{L-11}; however $\{t_1,t_2,x,v_1,v_2\}$ induces a hammer, a contradiction. 
	
	Suppose that $H$=house. It is certain that $v\ne p_2$ by (\ref{eq-3}). If $v\not\sim x$, then $\{t_1,t_2,x,v,p_2\}$ induces a bull whenever $v\not\sim p_2$, and $\{t_1,t_2,x,v,p_2\}$ induces a house whenever $v\sim p_2$, both are contradictions. Hence $v\sim x$. Since $v\notin V(S)$, we have that $v\notin M(V(C_2))$ by the choice of $S$. 
	Since $t_1\in M(V(C_2))$, we have that $v\notin V(C_2)$, which implies $v\in N(V(C_2))$. 
	
	Now there exists a triangle $t_1vxt_1$, an odd hole $C_2$, a vertex $t_1\in M(V(C_2))$, and an path $P'=vv_j$, where $v_j\in V(C_2)$ is a neighbor of $v$. By the choice of $C_1,C_2,P$, and $y$, we have that $|V(P)|\leq |V(P')|=2$, which implies that $x\in N(V(C_2))$. But the triangle $t_1vxt_1$ is the one that Claim~\ref{cc-2} asserts cannot exist, a contradiction. This completes the proof of Lemma~\ref{P-2}.
	\end{proof}
	

	\begin{theorem}\label{pollyanna-3}
		Let ${\cal C}$ be the class of $(\text{bull}, H)$-free graphs, where $H\in \{\text{house, hammer}\}$, ${\cal G}$ be a  hereditary class of graphs, and ${\cal H}$ be the class of triangle-free graphs. Then ${\cal G}\cap {\cal C}$ is perfectly weight divisible if and only if  ${\cal G}\cap {\cal H}$ is perfectly weight divisible.
	\end{theorem}
	\begin{proof}
		 It suffices to show that if each graph in ${\cal G}\cap {\cal H}$ is perfectly weight divisible, then each graph in ${\cal G}\cap {\cal C}$ is perfectly weight divisible. For the sake of contradiction, assume that there exists a MNWD graph $G$ in ${\cal G}\cap {\cal C}$. By Lemma \ref{P-2}, $G$ is triangle-free, and so $G\in {\cal G}\cap {\cal H}$, a contradiction. This proves Theorem~\ref{pollyanna-3}. 
	\end{proof}
	We are now ready to prove Theorems~\ref{evenhole} and~\ref{pollyanna-2}.
	
	\noindent\textbf{{\em Proof of Theorem~\ref{evenhole}:}} Notice that each even-hole-free graph is house-free. Let ${\cal G}$ be the class of even hole-free graphs, ${\cal C}$ be the class of $(\text{house, bull})$-free graphs, and ${\cal H}$ be the class of triangle-free graphs.  By Lemma \ref{P-000}, each graph of ${\cal G}\cap {\cal H}$ is 3-colorable, and, by Lemma \ref{3-coloring}, it is also perfectly weight divisible. By Theorem~\ref{pollyanna-3}, ${\cal G}\cap {\cal C}$ is perfectly weight divisible.\qed

	\noindent\textbf{{\em Proof of Theorem~\ref{pollyanna-2}:}}  By Lemma~\ref{3-coloring} and Theorem~\ref{pollyanna-3}, it is easy to verify that the Theorem~\ref{pollyanna-2} holds if $H\in\{\text{house, hammer}\}$. By Theorem~\ref{diamond}, we have that Theorem~\ref{pollyanna-2} holds if $H$=diamond. This completes the proof of Theorem~\ref{pollyanna-2}.  \qed

\section{Proof of Theorem~\ref{4K1}}\label{4}

In this section, we will prove the Theorem~\ref{4K1}. Let $G$ be a graph and $K\subseteq V(G)$ be a nonempty clique of $G$. We say that $K$ is a {\em clique cutset} of $G$ if $G-K$ has more components than $G$. {\em Substituting} a vertex $v$ of a graph $G$ by a graph $H$ is an operation which creates a new graph with vertex set $V(H)\cup V(G-v)$ and edge set $E(G-v)\cup \{xy~|~x\in V(H), y\in N_{G}(v)\}\cup E(H)$. When $H$ is a nonempty clique, the substitution is said to be {\em blowing up} $v$ of $G$ into a clique. A graph $G$ obtained from a graph $H$ by blowing up all the vertices into nonempty cliques is said to be a {\em clique blowup} of $H$.  Ho\`ang in \cite{H22} studied the structure and clique cutsets of MNPD graphs, and we prove the following lemma.

\begin{lemma}\label{cut}
	Let $G$ be a $4K_1$-free MNWD graph. Then $G$ has no clique cutset. 
\end{lemma} 
\begin{proof}
Suppose to the contrary that $K$ is a clique cutset of $G$. Let $(V_1,V_2)$ be a partition of $G-K$ such that $V_1$ is anticomplete to $V_2$. For $i\in\{1,2\}$, let $G_i=G[K\cup V_i]$. Since $G$ is $4K_1$-free, one of $V_1$ and $V_2$ is a clique. By symmetry, we may assume that $V_1$ is a clique, and so $G_1$ is perfect by Theorem~\ref{perfect}. 

Let $w$ be a positive integer weight function on $V(G)$. Since $G$ is MNWD, there exists a partition $(A',B')$ of $V(G_2)$ such that $G[A']$ is perfect and $\omega_w(G[B'])<\omega_w(G_2)\leq \omega_w(G)$. Let $A=A'\cup V_1$ and $B=B'$. It is easy to see that $(A,B)$ is a partition of $V(G)$. Then $G[A]$ is imperfect as $G$ is MNWD. 

Let $X_1=V_1$, $X_2=A\cap V_2$ and $S=A\cap K$. It is certain that $G[S\cup X_1]$ is perfect, and $G[S\cup X_2]$ is perfect (as $G[S\cup X_2]=G[A']$).  By Theorem~\ref{perfect},  we have that there is an odd hole or odd antihole $H$ in $G[A]$ such that $V(H)\cap X_i\ne\emptyset$ for $i\in \{1,2\}$. Since $S$ is a clique, $H$ is not an odd hole, and hence $H$ is an odd antihole. Clearly, $|V(H)|\geq7$. Let $x\in V(H)\cap X_1$ and $y\in V(H)\cap X_2$. If $V(H)\cap X_1=\{x\}$ and $V(H)\cap X_2=\{y\}$, then $|V(H)\cap S|\geq |V(H)|-2$, and so $H$ contains a clique of size $|V(H)|-2$, a contradiction. Hence $|V(H)\cap X_1|\ge 2$ or $|V(H)\cap X_2|\ge 2$.

By symmetry, we assume that $|V(H)\cap X_1|\geq2$. Since $y$ is anticomplete to $X_1$, $y$ has exactly two non-neighbors in $H$, and so $|V(H)\cap X_1|=2$. If $|V(H)\cap X_2|=1$, then $y$ is complete to $V(H)\cap S$, and so $V(H)\cap (S\cup X_2)$ is a clique of size $|V(H)|-2$, a contradiction. Hence, we assume that $|V(H)\cap X_2|\geq 2$. Since $x$ has exactly two non-neighbors in $H$, we have that $|V(H)\cap X_2|=2$. However, $V(H)\cap (S\cup X_1)$ is a clique of size $|V(H)|-2$, a contradiction. This proves Lemma~\ref{cut}. 
\end{proof}

\noindent\textbf{{\em Proof of Theorem~\ref{4K1}:}} Suppose to the contrary that $G$ is a (bull, $4K_1$)-free MNWD graph. Let $v\in V(G)$. By Lemma~\ref{MNWD} and Theorem~\ref{perfect}, $G$ is a connected locally perfect graph, and $G[M(v)]$ contains an odd hole. Since $G$ is $4K_1$-free, we have that $G[M(v)]$ contains a 5-hole $C=v_1v_2\cdots v_5v_1$.  From now on, the subscript is modulo 5. For $ i\in [5]$, let 
\begin{eqnarray*}
	X_i&=&\{u\in N(V(C))~|~N_{C}(u)=\{v_i\}\};\\	
	Y_i&=&\{u\in N(V(C))~|~N_{C}(u)=\{v_i,v_{i+2}\}\};\\
	Z_i&=&\{u\in N(V(C))~|~N_{C}(u)=\{v_{i-1},v_{i},v_{i+1}\}\};\\
	W_i&=&\{u\in N(V(C))~|~N_{C}(u)=\{v_i,v_{i+1},v_{i+2},v_{i+3}\}\}.
\end{eqnarray*}
Let $X=\bigcup_{i=1}^5X_i$, $Y=\bigcup_{i=1}^5Y_i$, $Z=\bigcup_{i=1}^5Z_i$ and $W=\bigcup_{i=1}^5W_i$. Let $M=M(V(C))$. Clearly, $v\in M$. We next prove the following claims.

\begin{claim}\label{cla-1}
	$V(G)=V(C)\cup X\cup Y\cup Z\cup W\cup M$.
\end{claim}
\begin{proof}
It suffices to prove that $N(V(C))\subseteq X\cup Y\cup Z\cup W$. Let $u\in N(V(C))$. By symmetry, we may assume that $u\sim v_1$. If $N_C(u)=\{v_1\}$, then $u\in X_1$. So we may assume that $u$ has a neighbor in $\{v_2,\cdots,v_5\}$. 

First, suppose that $u$ has a neighbor in $\{v_2,v_5\}$. By symmetry, we may assume that $u\sim v_2$. To avoid an induced bull on $\{u,v_1,v_2,v_3,v_5\}$, $u$ has a neighbor in $\{v_3,v_5\}$. By symmetry, suppose that $u\sim v_3$. Since $G$ is locally perfect, $u$ has a nonneighbor in $\{v_4,v_5\}$. Then $N_C(u)\in \{\{v_1,v_2,v_3\},\{v_1,v_2,v_3,v_4\},\{v_1,v_2,v_3,v_5\}\}$, and so $u\in Z\cup W$. 

Next, suppose that $u$ is anticomplete to $\{v_2,v_5\}$. Then $u$ has exactly one neighbor in $\{v_3,v_4\}$ as otherwise $\{u,v_3,v_4,v_2,v_5\}$ induces a bull. So $u\in Y$. This proves Claim~\ref{cla-1}.
\end{proof}
\begin{claim}\label{cla-2}
	$X\cup Y$ is complete to $M$, and $N(M)=X\cup Y$.
\end{claim}
\begin{proof}
Let $u\in X\cup Y$. Without loss of geneality, let $u\in X_1\cup Y_1$. To avoid an induced $4K_1$ on $\{u,v_2,v_4,t\}$, where $t\in M$, we have that $u$ is complete to $M$, and so $X\cup Y$ is complete to $M$. This implies $X\cup Y\subseteq N(M)$. 

By Claim~\ref{cla-1}, $N(M)\subseteq N(V(C))=X\cup Y\cup Z\cup W$. Suppose $N(M)\not\subseteq X\cup Y$. There exists a vertex $u'\in Z\cup W$ such that $u'$ has a neighbor in $M$, say $u''$. By symmetry, assume that $u'\in Z_2\cup W_1$. But $\{u',v_1,v_2,u'',v_5\}$ induces a bull, a contradiction. So Claim~\ref{cla-2} holds.
\end{proof}

\begin{claim}\label{cla-3}
	$M=\{v\}$, and so $V(G)=V(C)\cup X\cup Y\cup Z\cup W\cup\{v\}$.
\end{claim}
\begin{proof}
Suppose to the contrary that $|M|\geq2$. Then $M$ is a homogeneous set of $G$ by Claim~\ref{cla-2}, which contradicts Lemma~\ref{homogeneous}. So Claim~\ref{cla-3} holds. 
\end{proof}

Let $i\in [5]$. It is certain that both $X_i$ and $Y_i$ are cliques as $G$ is $4K_1$-free. By Claim~\ref{cla-2}, we can deduce that both $Z_i$ and $W_i$ are cliques. We will prove the following claims. 
\begin{claim}\label{cla-4}
	Both $X_i\cup Y_i$ and $Z_{i+1}\cup W_i$ are cliques.
\end{claim}
\begin{proof}
If there exist $u,u'\in X_i\cup Y_i$ such that $u\not\sim u'$, then $\{u,u',v_{i+1},v_{i+3}\}$ induces a $4K_1$. If there exist $t,t'\in Z_{i+1}\cup W_i$ such that $t\not\sim t'$, then $\{t,t',v,v_{i+4}\}$ induces a $4K_1$. Both are contradictions. This proves Claim~\ref{cla-4}. 
\end{proof}
\begin{claim}\label{cla-5}
	 $X=X_j\cup X_{j+1}$ for some $j\in [5]$, and $X$ is a clique.
\end{claim}
\begin{proof}
We first show that
 either $X_i=\emptyset$ or $X_{i+2}=\emptyset$. Suppose to the contrary that $X_i\neq \emptyset$ and $X_{i+2}\neq \emptyset$. Without loss of generality, we may assume that $x_1\in X_1$ and $x_3\in X_3$. By Claims~\ref{cla-2} and \ref{cla-3}, $v$ is complete to $\{x_1,x_3\}$. If $x_1\sim x_3$, then $\{v,x_1,x_3,v_1,v_3\}$ induces a bull; otherwise, $\{x_1,x_3,v_2,v_4\}$ induces a $4K_1$. Therefore, $X=X_j\cup X_{j+1}$ for some $j\in [5]$. Let $x_j\in X_j$ and $x_{j+1}\in X_{j+1}$. If $x_j\not\sim x_{j+1}$, then $\{x_j,x_{j+1},v_{j+2},v_{j+4}\}$ induces a $4K_1$, a contradiction. Hence, $X_j$ is complete to $X_{j+1}$. So, $X$ is a clique.
\end{proof}

\begin{claim}\label{cla-6}
	Either $Y_i=\emptyset$ or $(X\setminus X_{i+1})\cup Y_{i+2}\cup Y_{i+3}\cup (W\setminus W_{i+2})=\emptyset$.
\end{claim}
\begin{proof} 
	On the contrary, without loss of generality, let $y\in Y_1$ and $x\in (X\setminus X_{2})\cup Y_{3}\cup Y_{4}\cup (W\setminus W_{3})$. If $x\in X_1$, then $x\sim y$ by Claim~\ref{cla-4}, and now $\{x,y,v_1,v_3,v_5\}$ induces a bull. So $x\not\in X_1$. Similarly, $x\not\in X_3$. If $x\in X_4$, then $\{x,y,v_2,v_5\}$ induces a $4K_1$ whenever $x\not\sim y$, and $\{v,x,y,v_1,v_4\}$ induces a bull whenever $x\sim y$ by Claim~\ref{cla-2}. So $x\not\in X_4$. Similarly, $x\not\in X_5$. Therefore, $x\not\in X\setminus X_2$. 

If $x\in Y_3$, then $\{x,y,v_2,v_4\}$ induces a $4K_1$ whenever $x\not\sim y$, and $\{x,y,v_3,v_2,v_5\}$ induces a bull whenever $x\sim y$. So $x\not\in Y_3$. Similarly, $x\not\in Y_4$. Therefore, $x\in W\setminus W_{3}$, and so $x\in W_1\cup W_2\cup W_4\cup W_5$.

If $x\in W_1$, then $\{x,v_3,v_4,y,v_5\}$ induces a bull whenever $x\not\sim y$, and $\{x,y,v_1,v,v_5\}$ induces a bull whenever $x\sim y$. So $x\not\in W_1$. Similarly, $x\not\in W_5$. If $x\in W_2$, then $\{x,v_2,v_3,y,v_5\}$ induces a bull whenever $x\not\sim y$, and $\{x,y,v_3,v,v_5\}$ induces a bull whenever $x\sim y$ by Claim~\ref{cla-2}. So $x\not\in W_2$. Similarly, $x\not\in W_4$. This proves Claim~\ref{cla-6}. 
\end{proof}
\begin{claim}\label{cla-7}
	$Y_i$ is anticomplete to $Y_{i+1}\cup Z_{i+1}\cup Z_{i+3}\cup Z_{i+4}\cup W_{i+2}$.
\end{claim}
\begin{proof}
On the contrary, without loss of generality, let $y\in Y_1$ and $x\in Y_2\cup Z_2\cup Z_4\cup Z_5\cup W_3$ such that $x\sim y$. If $x\in Y_2$, then $\{v,x,y,v_1,v_4\}$ induces a bull by Claim~\ref{cla-2}. If $x\in Z_2$, then $\{x,y,v_3,v,v_4\}$ induces a bull by Claim~\ref{cla-2}. If $x\in Z_4\cup W_3$, then $\{x,y,v_3,v,v_2\}$ induces a bull by Claim~\ref{cla-2}. Similarly, we can obtain a contradiction if $x\in Z_5$. This proves Claim~\ref{cla-7}. 
\end{proof}
\begin{claim}\label{cla-8}
	$Y_i\cup Z_{i+1}$ is complete to $Z_{i}$, and $Y_i$ is complete to $Z_{i+2}$.
\end{claim}
\begin{proof}
	 By symmetry, it suffices to prove that $Y_i\cup Z_{i+1}$ is complete to $Z_{i}$.
On the contrary, without loss of generality, let $y\in Y_1\cup Z_2$ and $x\in Z_1$ with $x\not\sim y$. Then $\{x,v_1,v_5,y,v_4\}$ induces a bull. This proves Claim~\ref{cla-8}. 
\end{proof}

By Claims~\ref{cla-2} and \ref{cla-5}, $N(M)=X\cup Y$ and $X$ is a clique. By Lemma~\ref{cut}, we have that $Y\ne\emptyset$ as otherwise $X$ is a clique cutset of $G$. Without loss of generality, we may assume $Y_1\ne\emptyset$. By Claim~\ref{cla-6}, we have that $Y_3\cup Y_4=\emptyset$ and either $Y_2=\emptyset$ or $Y_5=\emptyset$. Without loss of generality, we may assume that $Y=Y_1\cup Y_2$. By Claim~\ref{cla-6} again, $X=X_2$ and $W=W_3$. Since $Y_1\ne\emptyset$, we may assume that $y_1\in Y_1$. By Claim~\ref{cla-3}, we have that 
\begin{equation}\label{e-11}
	\mbox{$V(G)=V(C)\cup X_2\cup Y_1\cup Y_2\cup Z\cup W_3\cup \{v\}$ and $y_1\in Y_1$. }
\end{equation}

\begin{claim}\label{cla-9}
	$G[V(C)\cup Z]$ is a clique blowup of 5-hole.
\end{claim}
\begin{proof}
By Claims~\ref{cla-4} and \ref{cla-8}, for $i\in [5]$, $Z_i$ is complete to $Z_{i+1}$ and $Z_i\cup \{v_{i}\}$ is a clique.

 Suppose that there is an edge $z_1z_3$ such that $z_i\in Z_i$ for $i\in\{1,3\}$. By Claim \ref{cla-2}, $v\sim y_1$ and $v$ is anticomplete to $\{z_1,z_3\}$. By Claim \ref{cla-8}, $y_1$ is complete to $\{z_1,z_3\}$.  Then $\{z_1,y_1,z_3,v,v_4\}$ induces a bull, a contradiction. So $Z_1$ is anticomplete to $Z_3$.
 
 Suppose that there is an edge $z_1z_4$ such that $z_i\in Z_i$ for $i\in\{1,4\}$. By Claim \ref{cla-2}, $v\sim y_1$ and $v$ is anticomplete to $\{z_1,z_4\}$. By Claims \ref{cla-7} and \ref{cla-8}, $y_1\sim z_1$ and $y_1\not\sim z_4$. Then $\{v_1,y_1,z_1,v,z_4\}$ induces a bull, a contradiction. So $Z_1$ is anticomplete to $Z_4$. Similarly, we have that $Z_3$ is anticomplete to $Z_5$. 
 
 Suppose that there is an edge $z_2z_4$ such that $z_i\in Z_i$ for $i\in\{2,4\}$. By Claim~\ref{cla-7}, $y_1$ is anticomplete to $\{z_2,z_4\}$. Then $\{v_2,v_1,z_2,y_1,z_4\}$ induces a bull, a contradiction. So $Z_2$ is anticomplete to $Z_4$. Similarly, we have that $Z_2$ is anticomplete to $Z_5$. This proves Claim~\ref{cla-9}. \end{proof}

 Suppose that $Y_2\ne\emptyset$. By Claim~\ref{cla-6}, $X_2\cup W_3=\emptyset$, and so $V(G)=V(C)\cup Y_1\cup Y_2\cup Z\cup \{v\}$ by (\ref{e-11}). By Claims~\ref{cla-7}, \ref{cla-8}, and \ref{cla-9}, $G$ is a clique blowup of $F$ (See Figure~\ref{fig-2}). Since every MNWD graph has no homogeneous set by Lemma~\ref{homogeneous}, $G$ is isomorphic to $F$. Then $\chi(G)\leq3$ and $G$ is triangle-free, which implies that $G$ is perfectly weight divisible by Lemma~\ref{3-coloring}, a contradiction. Therefore, $Y_2=\emptyset$. Now $V(G)=V(C)\cup X_2\cup Y_1\cup Z\cup W_3\cup \{v\}$ by (\ref{e-11}).
 
 	\begin{figure}[htbp]
 	\begin{center}
 		\includegraphics[width=15cm]{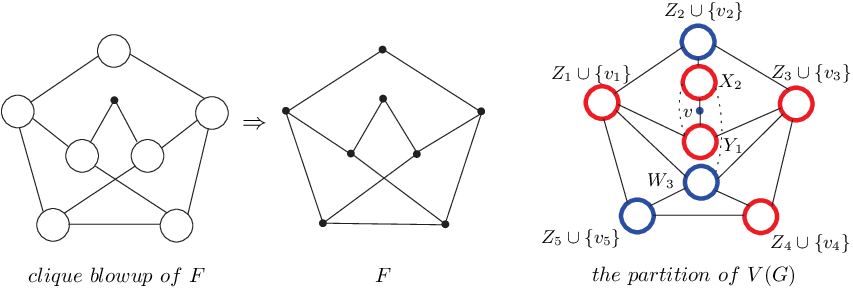}
 	\end{center}
 	\vskip -25pt
 	\caption{Illustration of $F$, and the partition of $V(G)$: the red sets belong to $A$ and the blue sets belong to $B$.}
 	\label{fig-2}
 \end{figure}
 
 \begin{claim}\label{cla-10}
 	$X_2$ is anticomplete to $Z\setminus Z_2$ and complete to $Z_2$.
 \end{claim}
\begin{proof}
 Let $x\in X_2$ and $z\in Z_3\cup Z_4$ such that $x\sim z$. If  $z\in Z_3$, then $\{x,z,v_2,v_1,v_4\}$ induces a bull.  If $z\in Z_4$, then $\{z,v_4,v_5,x,v_1\}$ induces a bull. Both are contradictions. So $X_2$ is anticomplete to $Z_3\cup Z_4$. Similarly, $X_2$ is anticomplete to $Z_1\cup Z_5$. This implies $X_2$ is anticomplete to $Z\setminus Z_2$. 
 
 Let $x'\in X_2$ and $z'\in Z_2$. If $x'\not\sim z'$. Then $\{z',v_2,v_3,x',v_4\}$ induces a bull, a contradiction. This implies $X_2$ is complete to $Z_2$ and proves Claim~\ref{cla-10}. 
 \end{proof}
 
\begin{claim}\label{cla-11}
 $W_3$ is anticomplete to $Z_2$ and complete to $Z_1\cup Z_3$.
\end{claim}
\begin{proof}
 Let $t\in W_3$ and $z\in Z_2$. Suppose that $t\sim z$. By Claim \ref{cla-7}, $y_1$ is anticomplete to $\{t,z\}$.
 Then $\{z,t,v_3,v_5,y_1\}$ induces a bull, a contradiction. So $W_3$ is anticomplete to $Z_2$.
 
 Suppose that $W_3$ is not complete to $Z_1\cup Z_3$. Without loss of generality, let $t'\in W_3$ and $z'\in Z_1$ such that $t'\not\sim z'$. However, $\{v_1,t',z',v,y_1\}$ induces a bull by Claims \ref{cla-7} and \ref{cla-8}, a contradiction. This proves Claim~\ref{cla-11}.
\end{proof}

Let $w$ be any positive integer weight function on $V(G)$, and let\\[-24pt]
\begin{itemize}
	\item $A=X_2\cup Y_1\cup Z_1\cup Z_3\cup Z_4\cup \{v_1,v_3,v_4\}$;\\[-24pt]
	\item $B=Z_2\cup Z_5\cup W_3\cup \{v,v_2,v_5\}$.
\end{itemize}
By (\ref{e-11}), $(A,B)$ is a partition of $V(G)$ (See Figure~\ref{fig-2}). We next prove the following two claims.

\begin{claim}\label{cla-12}
	$G[A]$ is perfect.
\end{claim}
\begin{proof}
Since $X_2$ and $Y_1$ are two cliques by Claim~\ref{cla-4}, $G[X_2\cup Y_1]$ is perfect by Theorem~\ref{perfect}. For every vertex $x\in Z_1\cup \{v_1\}$, $N(x)\cap (X_2\cup Y_1\cup Z_1\cup \{v_1\})$ is a clique by Claims~\ref{cla-8}, \ref{cla-9}, and \ref{cla-10}, and so $G[X_2\cup Y_1\cup Z_1\cup \{v_1\}]$ is perfect. By Claims~\ref{cla-8}, \ref{cla-9}, and \ref{cla-10} again, for every vertex $x\in Z_3\cup \{v_3\}$, $N(x)\cap (X_2\cup Y_1\cup Z_1\cup Z_3\cup \{v_1,v_3\})$ is a clique, and so $G[X_2\cup Y_1\cup Z_1\cup Z_3\cup \{v_1,v_3\}]$ is perfect. By Claims~\ref{cla-7}, \ref{cla-9}, and \ref{cla-10}, for every vertex $x\in Z_4\cup \{v_4\}$, $N(x)\cap A$ is a clique, and so $G[A]$ is perfect. This proves Claim~\ref{cla-12}.
\end{proof}

\begin{claim}\label{cla-13}
	$\omega_w(G[B])<\omega_w(G)$.
\end{claim}
\begin{proof}
It suffices to prove that for any maximal clique $K$ of $G[B]$, $K$ is not a maximal clique of $G$. Suppose to the contrary that $K\subseteq B$ and $K$ is a maximal clique of $G$. Since $v$ is anticomplete to $B\setminus \{v\}$ and $v$ is complete to $Y_1$ (Recall that $Y_1\ne\emptyset$)  by Claim~\ref{cla-2}, it follows that $v\notin K$. Since $Z_2\cup \{v_2\}$ is anticomplete to $B\setminus (Z_2\cup \{v_2\})$ by Claims~\ref{cla-2}, \ref{cla-9}, and \ref{cla-11}, we have that $K\cap (Z_2\cup \{v_2\})=\emptyset$ as $v_1$ is complete to $Z_2\cup \{v_2\}$. Therefore $K\subseteq  Z_5\cup W_3\cup \{v_5\}$. Then $v_4$ is complete to $K$, a contradiction. This proves Claim~\ref{cla-13}.
\end{proof}

Combining Claims~\ref{cla-12} and \ref{cla-13}, we can obtain a contradiction as $G$ is MNWD. This completes the proof of Theorem~\ref{4K1}. \qed

\section{Conclusions}

  Chen and Xu \cite{CX2025} proved that every (bull, $P_7,C_5$)-free graph is perfectly divisible. Let ${\cal G}$ be the class of ${\cal H}$-free graphs, where ${\cal H}\in$\{\{odd torch\}, \{even hole\},\{$4K_1$\}, \{$P_7,C_5$\}, \{$P_{11},C_4$\}, \{$P_{14}, C_{5},C_4$\}, \{$P_{17},C_6,C_5,C_4$\}\} and, let ${\cal C}$ be the class of bull-free graphs. In this paper, we show that the class of (bull, $H$)-free graphs is perfect-Pollyanna, where $H\in\{\text{house, hammer, diamond}\}$. Notice that every triangle-free graph in ${\cal G}$ is 3-colorable. By Theorems~\ref{oddballoon}, \ref{evenhole}, \ref{4K1} and Corollary~\ref{co-2}, we have that ${\cal G}\cap {\cal C}$ is perfectly divisible. These are two pieces of evidence suggesting that the class of bull-free graphs may be perfect-Pollyanna. Therefore, there are two interesting problems as follows (It is easy to verify that if the second problem is true, then the first problem is true.):
  
  \begin{problem}\label{p1}
  	Is it true that the class of bull-free graphs is perfect-Pollyanna$?$
  \end{problem}
  
  \begin{problem}\label{p2}
  	Is it true that every bull-free  MNWD graph is triangle-free$?$
  \end{problem}

	{	\section*{Declarations}
		\begin{itemize}
			\item Ran Chen is supported by Postgraduate Research and Practice Innovation Program of Jiangsu
			Province KYCX25\_1926. Di Wu is supported by   the Scientific Research
			Foundation of Nanjing Institute of Technology, China (No. YKJ202448).
			\item \textbf{Conflict of interest}\quad The authors declare no conflict of interest.
			\item \textbf{Data availibility statement}\quad This manuscript has no associated data.
	\end{itemize}} 
	
	{
	}

	\end{document}